\documentclass[12pt]{amsart}
\usepackage{amssymb,amsthm,amsmath,amstext}
\usepackage{mathrsfs}  
\usepackage{mathtools} 
\usepackage{colonequals}

\usepackage{multirow}
\usepackage{fullpage}

\numberwithin{equation}{section}

\theoremstyle{plain}
\newtheorem{theorem}[equation]{Theorem}
\newtheorem{proposition}[equation]{Proposition}
\newtheorem{prop}[equation]{Proposition}
\newtheorem{lemma}[equation]{Lemma}
\newtheorem{lem}[equation]{Lemma}

\newtheorem{cor}[equation]{Corollary}

\theoremstyle{definition}
\newtheorem{defin}[equation]{Definition}

\newtheorem{exm}[equation]{Example}

\theoremstyle{remark}
\newtheorem{remark}[equation]{Remark}
\newtheorem{rmk}[equation]{Remark}

\newtheorem{ques}[equation]{Question}

\newcommand{\Z}{\mathbb Z}

\newcommand{\C}{\mathbb C}

\renewcommand{\P}{\mathbb P}
\newcommand{\Q}{\mathbb Q}
\newcommand{\R}{\mathbb R}

\newcommand{\HH}{\mathbb H}

\newcommand{\matT}{\textup{\textsf{T}}}

\DeclareMathOperator{\disc}{\mathrm{disc}}

\DeclareMathOperator{\GL}{\mathrm{GL}}
\DeclareMathOperator{\M}{\mathrm{M}}

\DeclareMathOperator{\SL}{\mathrm{SL}}
\DeclareMathOperator{\sgn}{\mathrm{sgn}}

\newcommand{\calO}{\mathcal{O}}

\newcommand{\id}{\mathrm{id}}

\DeclareMathOperator{\Tr}{\mathrm{Tr}}
\DeclareMathOperator{\pf}{\mathrm{pf}}

\DeclareMathOperator{\adj}{\mathrm{adj}}
\DeclareMathOperator{\discpf}{\mathrm{discpf}}
\DeclareMathOperator{\dpf}{\discpf}

\newcommand{\linedef}[1]{\textsf{#1}}
\newcommand{\defi}[1]{\linedef{#1}}

\usepackage[backref=page]{hyperref}
\hypersetup{pdftitle={Stickelberger's discriminant theorem for algebras}}
\hypersetup{pdfauthor={Asher Auel, Owen Biesel, John Voight}}
\hypersetup{colorlinks=true,linkcolor=blue,anchorcolor=blue,citecolor=blue}

\begin{document}

\title[Stickelberger's discriminant theorem]{Stickelberger's discriminant theorem for algebras}

\author{Asher Auel}
\address{Department of Mathematics, Dartmouth College, Kemeny Hall, Hanover, NH 03755, USA}
\email{asher.auel@dartmouth.edu}

\author{Owen Biesel}
\address{Department of Mathematics and Statistics, Carleton College, Center for Mathematics and Computing, Northfield, MN 55057, USA}
\email{owenbiesel@gmail.com}

\author{John Voight}
\address{Department of Mathematics, Dartmouth College, Kemeny Hall, Hanover, NH 03755, USA}
\email{jvoight@gmail.com}

\date{\today}

\begin{abstract}
Stickelberger proved that the discriminant of a number field is
congruent to $0$ or $1$ modulo $4$.  We generalize this to an
arbitrary (not necessarily commutative) ring of finite rank over $\Z$
using techniques from linear algebra.  Our proof relies on elementary matrix
identities.
\end{abstract}

\maketitle

\section{Introduction}

The \emph{discriminant} arises naturally in many situations in
mathematics, often as a measure of size or arithmetic complexity.  In
perhaps its simplest form, we learn that a quadratic equation
$ax^2+bx+c=0$ with $a,b,c \in \R$ has a real root if and only if its
discriminant $d \colonequals b^2-4ac$ is nonnegative.  In algebraic
number theory, the discriminant of a number field measures
ramification of primes \cite[Chapters 2--3]{Marcus:nf}; in the theory
of differential equations, the discriminant measures the extent to
which singular solutions exist.

In this note, we pursue discriminants in the context of rings and with
a view toward arithmetic.  

\subsection*{Motivation}

As motivation, we consider a very simple case: 
let $d \in \Z$ be a nonsquare and consider the quadratic ring
\begin{equation} 
\Z[\sqrt{d}] \colonequals \{a+b\sqrt{d} : a,b \in \Z\} \subseteq \C. 
\end{equation}
This ring has a natural notion of \defi{trace} given by 
\[
\Tr(a+b\sqrt{d})=(a+b\sqrt{d})+(a-b\sqrt{d})=2a \in \Z.
\]  
Of course $\Z[\sqrt{d}] = \Z + \Z \sqrt{d} \simeq \Z^2$ as abelian groups, and multiplication in $\Z[\sqrt{d}]$ can be written out as
\begin{equation} \label{eqn:multout}
(a+b\sqrt{d})(a'+b'\sqrt{d}) = (aa'+bb'd) + (ab'+a'b)\sqrt{d}
\end{equation}
for $a,b,a',b' \in \Z$.  The multiplication law \eqref{eqn:multout} in
$\Z[\sqrt{d}]$ can be given without an embedding into $\C$: on the
free abelian group $\Z^2$ with basis $1,e$, there is a unique ring
structure satisfying $e^2=d$.  Indeed, by the distributive law, it is
enough to remember the products of basis elements, with only the
product $e \cdot e$ needing to be specified.  Finally, we can recover the
discriminant from the \emph{traces} of these products, taking the
determinant:
\begin{equation} 
\det\begin{pmatrix} \Tr(1\cdot 1) & \Tr(1 \cdot e) \\ \Tr(e \cdot 1) & \Tr(e \cdot e) \end{pmatrix} =
\det\begin{pmatrix} 2 & 0 \\ 0 & 2d \end{pmatrix} = 4d. 
\end{equation}
This calculation agrees with the more familiar notion, since $\sqrt{d}$ is a root of the equation $x^2-d=0$ which has discriminant $4d$.  In a similar manner, we can define a ring structure for $e$ satisfying $e^2+be+c=0$, and we find the discriminant $b^2-4c$.

This approach works more generally.  Let $K$ be a number field (a finite extension of $\Q$), and let $\Z_K$ be its ring of integers, the subset of $K$ of elements that satisfy a monic polynomial with integer coefficients \cite[Chapter 1]{Marcus:nf}.  For example, we might take $K=\Q(\sqrt{-1})$, in which case $\Z_K=\Z[i] = \{a+bi : a,b \in \Z\}$.  Then one can define the discriminant of $\Z_K$ in a similar manner: if $\alpha_1,\dots,\alpha_n$ is an integral basis for $\Z_K$, and $\Tr \colon K \to \Q$ the trace, then we form the $n \times n$-matrix
\begin{equation}  \label{eqn:BTralp}
B \colonequals ( \Tr(\alpha_i \alpha_j) )_{i,j=1}^n \in \M_n(\Z)
\end{equation}
and define the discriminant
\[ \disc \Z_K \colonequals \det B. \]
(See Remark \ref{rmk:PNPN} for an equivalent definition in the context of Minkowski's \emph{geometry of numbers}.)
The matrix $B$ can be interpreted in linear algebraic terms: the bilinear form
\begin{equation}
\begin{aligned}
\Tr \colon K \times K &\to \Q \\
(\alpha,\beta) &\mapsto \Tr(\alpha\beta) 
\end{aligned}
\end{equation}
is symmetric (and nondegenerate), and the matrix $B$ is the \emph{Gram matrix} of this bilinear form in the basis $\alpha_1,\dots,\alpha_n$.

Visibly, for quadratic rings we have $b^2-4c \equiv b^2 \equiv 0,1 \pmod{4}$.  In fact, this congruence generalizes to all rings of integers, the starting point of our investigation.

\begin{theorem}[Stickelberger]
We have $\disc \Z_K \equiv 0,1 \pmod{4}$.
\end{theorem}

This theorem is called \emph{Stickelberger's discriminant theorem}, among other names.  While never stated explicitly in
Stickelberger's work \cite{stickelberger:discriminant}, this statement can be deduced from the main
results.  The modern simple proof given by Schur \cite{schur:stickelberger} is typically provided as an exercise in an algebraic number theory class (see e.g.\ Marcus \cite[Chapter 2, Exercise 22]{Marcus:nf} or Neukirch \cite[Section I.2, Exercise~7]{neukirch}).  For further discussion, see Remark \ref{rmk:PNPN}; and for more on this history, see Cox \cite{cox:stickelberger}.  (There is a different, much deeper, theorem of Stickelberger in algebraic number theory that describes the Galois module structure of class groups of cyclotomic fields.  For more on this theorem, see Washington \cite[Chapter 6]{Washington}.)  Various generalizations of this congruence have also been made \cite{martinet:discriminant,berlekamp:discriminant,baeza:discriminant,harrer,biesel_gioia:discriminant}.

\subsection*{Generalization}

With the motivation to study discriminants as measuring the bilinear form coming from the trace of multiplication, we are now ready to generalize. 
A \defi{ring of rank $n \in \Z_{\geq 1}$} is a ring (with $1$), not necessarily commutative,
whose underlying additive group is isomorphic to $\Z^n$.  Concretely, in a $\Z$-basis
$e_1, e_2, \dots, e_n$ for $A \simeq \Z^n$, multiplication is defined by
\begin{equation} \label{eqn:multtable}
e_i e_k = \sum_{j=1}^n c_{ijk} e_j 
\end{equation}
for $i,k=1,\dots,n$, with $c_{ijk} \in \Z$ (with multiplication extended to $A$ using the
distributive law).  The $n^3$ coefficients
 $(c_{ijk})_{i,j,k=1}^n$ form what is called a \defi{multiplication table} for $A$.  

Commutative rings of rank $n$, including rings of integers in number
fields, are of considerable interest.  For an overview, see Bhargava
\cite{bhargava:ICM}.  However, we do not restrict our work here to the
commutative case.  Already, the ring $\M_n(\Z)$ of $n \times
n$-matrices with entries in $\Z$ is a ring of rank $n^2$,
noncommutative for $n \geq 2$.  

Other noncommutative examples of rings of rank $n$ abound.  Even before J.J.\ Sylvester
coined the term ``matrix'' in 1848, Sir William Rowan Hamilton had discovered in 1843 the noncommutative algebra of quaternions
\[ \HH \colonequals \R + \R i + \R j + \R k \]
famously inscribing the equations
\[ i^2=j^2=k^2=ijk=-1 \]
into the Broom Bridge in Dublin.  Fifty years later,
Hurwitz \cite{Hurwitz:uberdie} considered the subring of (integral) \defi{Hurwitz quaternions}
\begin{equation} 
\mathcal{O} \colonequals \left\{ t+xi+yj+zk \in \HH : 
\begin{minipage}{36ex} 
\begin{center}
\textup{$t,x,y,z \in \tfrac{1}{2}\Z$ and } \\
\textup{$2t,2x,2y,2z \in \Z$ of the same parity}
\end{center}
\end{minipage} \right\}.
\end{equation}
A $\Z$-basis for $\calO$ is given by $1,i,j,\omega$ where $\omega
\colonequals (-1+i+j+k)/2$ satisfies the identity $\omega^2+\omega+1=0$.  The ring $\calO$ is a noncommutative ring of rank $4$; it may be thought of as a noncommutative analogue of the ring of integers of a quadratic field.  (For further reading, see Voight \cite{Voight:quat}.)

In fact, \emph{every} ring $A$ of rank $n$ is a subring of $\M_n(\Z)$. 
Explicitly, the coefficients of the multiplication table \eqref{eqn:multtable} provide a map
\begin{equation} \label{eqn:lambdamult}
\begin{aligned}
\lambda \colon A &\to \M_n(\Z) \\
e_i &\mapsto (c_{ijk})_{j,k=1,\dots,n} 
\end{aligned}
\end{equation}
(extended $\Z$-linearly) which defines an injective ring homomorphism.  Analogously to the above, we then define the \defi{discriminant} of $A$ by
\begin{equation} 
\disc(A) \colonequals \det(B) 
\end{equation}
where $B=(b_{ij})_{i,j} \in \M_n(\Z)$ is the matrix obtained by taking the trace of pairwise products of basis elements
\begin{equation}\label{def-bij}
b_{ij} \colonequals \Tr(\lambda(e_ie_j)).
\end{equation}
The discriminant $\disc(A)$ does not depend on the basis (see Lemma \ref{lem:welldefdisc}).

\begin{exm}
Computed using the basis of matrix units, we have $\disc(\M_n(\Z)) = (-1)^{n(n-1)/2} n^{n^2}$.
\end{exm}

\begin{exm}
For the Hurwitz quaternions $\mathcal{O}$ in the basis $1,i,j,\omega$, we have for example
\[ 
\lambda(i) = \begin{pmatrix} 
0 & -1 & 1 & 0 \\
1 & 0 & -1 & -1 \\
0 & 0 & -1 & -1 \\
0 & 0 & 2 & 1 
\end{pmatrix} \]
since $ij=k=1-i-j+2\omega$ and $i\omega = -i - j + \omega$.  Multiplying matrices and taking traces yields
\[ 
B = \begin{pmatrix}
4 & 0 & 0 & -2 \\
0 & -4 & 0 & -2 \\
0 & 0 & -4 & -2 \\
-2 & -2 & -2 & -2
\end{pmatrix} \]
and we find that $\disc(A) = \det(B) = -64$.
\end{exm}

\subsection*{Main result}

Our main result is a generalization of Stickelberger's theorem to an arbitrary rank $n$ ring.

\begin{theorem}
If $A$ is a ring of rank $n$, then $\disc(A) \equiv 0,1 \pmod{4}$.  
\end{theorem}

We prove this theorem using purely linear algebra techniques (as Theorem \ref{main-theorem}), 
giving a new proof of Stickelberger's theorem even in the case of the ring of integers of a number field.  Moreover, our proof
introduces a new invariant of a ring of rank $n$ equipped with a basis $\beta$ containing $1$.  We call it the
\linedef{discriminant pfaffian} $\dpf(A,\beta) \in \Z$ (see \S \ref{sec:discpf}), and it satisfies
\[ \disc(A) \equiv \dpf(A,\beta)^2 \pmod{4}. \]
For a quadratic ring $A \colonequals \Z[x]/(x^2-bx+c)$ we have $\dpf(A,(1,e))=b$.  And in general our discriminant pfaffian extracts a square root of the ``square part'' of the discriminant modulo $4$.  The name is motivated by the analogy with the classical pfaffian, the square root of the determinant of a skew-symmetric matrix.

\begin{exm}
Let $A$ be a ring of rank $3$, with basis $(1,e_2,e_3)$.  Let $B=(b_{ij})_{i,j}$.  Then $\discpf(A,\beta)=b_{12}b_{13}+b_{23}=\Tr(\lambda(e_2))\Tr(\lambda(e_3))+\Tr(\lambda(e_2e_3))$.
\end{exm}

\subsection*{Organization}

This paper is organized as follows.  In Section \ref{sec:notation} we set up background and notation.  In Section \ref{sec:over-Z} we prove our main result, and then in Section \ref{sec:discpf} we describe the discriminant pfaffian.  

\subsection*{Acknowledgments}
The authors would like to thank Darij Grinberg for posing the question
\cite{Darij}, for helpful correspondence, and for feedback.  The
authors are also grateful to the reviewers for their comments.  Auel
was supported by a Simons Foundation Collaboration Grant (712097),
a National Science Foundation Grant (2200845), and a Walter and Constance Burke Research Award.  Voight was supported by a Simons Collaboration Grant (550029).

\section{Notation}\label{sec:notation}

We begin by setting notation, building upon and detailing what was presented in the introduction.  Throughout this paper, by a \defi{ring} we mean a (not necessarily commutative) ring with multiplicative identity $1$.

\begin{defin}
Let $n \in \Z_{\geq 1}$.  A \defi{ring of rank $n$} is a ring that is isomorphic to $\Z^n$ as a $\Z$-module (equivalently, as an abelian group).
\end{defin}

\begin{defin} \label{defn:basisring}
Let $A$ be a ring of rank $n$.  A \defi{basis} for $A$ is an ordered $n$-tuple $\beta=(e_1,\dots,e_n)$ of elements of $A$ that generate $A$ as a $\Z$-module.  The \defi{multiplication table} for $A$ in a basis $\beta$ is the tuple $(c_{ijk})_{i,j,k}$ of $n^3$ coefficients $c_{ijk} \in \Z$ defined by
\begin{equation}
e_i e_k = \sum_{j=1}^n c_{ijk} e_j. 
\end{equation}

A \defi{framed ring} $(A,\beta)$ \defi{of rank $n$} is a ring $A$ of rank $n$ equipped with a basis $\beta$.  
\end{defin}


Let $(A,\beta)$ be a framed ring of rank $n$ with $\beta=(e_1,\dots,e_n)$.  

\begin{defin}
The \defi{matrix} of $a\in A$ is $\lambda_\beta(a)=(a_{ij})_{i,j} \in \M_n(\Z)$ where
 \[ ae_j = \sum_{i=1}^n a_{ij} e_i. \]
\end{defin}

The following lemma follows from a direct verification.

\begin{lemma}
The matrix map
\[ \lambda_\beta \colon A \hookrightarrow \M_n(\Z) \]
defines an injective ring homomorphism, and the map 
\begin{equation} \label{eqn:trace}
\begin{aligned}
t_\beta \colon A \times A &\to \Z \\
(a,b) &\mapsto \Tr(\lambda_\beta(ab))
\end{aligned}
\end{equation}
defines a symmetric, bilinear pairing on $A$.
\end{lemma}


The matrix (or ``left multiplication'') map $\lambda_\beta$ is called the \defi{left regular representation} of $A$.  Indeed, on basis elements we have $\lambda_\beta(e_i)$ defined by the entries of the multiplication table as in \eqref{eqn:lambdamult}.  We call the map $t_\beta$ the \defi{trace pairing} on $A$.

\begin{defin}\label{discriminant}
The \defi{Gram matrix} of $(A,\beta)$ is the (symmetric) matrix $B=B(A,\beta)$ defined by
\[ (t_\beta(e_i,e_j))_{i,j=1,\dots,n} = \begin{pmatrix}
t_\beta(e_1,e_1) & t_\beta(e_1,e_2) & \dots & t_\beta(e_1,e_n) \\
t_\beta(e_2,e_1) & t_\beta(e_2,e_2) & \dots & t_\beta(e_2,e_n) \\
\vdots & \vdots & \ddots & \vdots \\
t_\beta(e_n,e_1) & t_\beta(e_n,e_2) & \dots &
t_\beta(e_n,e_n) \end{pmatrix}. \]
Thus the Gram matrix of $(A,\beta)$ is defined to be the classical
Gram matrix of the trace form $t_\beta$. The \defi{discriminant} of
$A$ (with respect to $\beta$) is
\[ \disc(A,\beta) \colonequals \det(B). \]
\end{defin}

\begin{lem} \label{lem:welldefdisc}
The trace pairing $t=t_\beta$ and the discriminant $\disc(A)=\disc(A,\beta)$ are well-defined, independent of the choice of basis $\beta$.
\end{lem}

\begin{proof}
Let $Q=[\id]_{\beta}^{\beta'} \in \GL_n(\Z)$ be a change of basis from $\beta$ to $\beta'$.  Then $\lambda_{\beta'}(a)=Q\lambda_{\beta}(a)Q^{-1}$ so $\Tr(\lambda_{\beta'}(a))=\Tr(\lambda_{\beta}(a))$ for all $a \in A$, hence $t$ is independent of the choice of basis.  Correspondingly, we have $B(A,\beta')=Q^{\matT} B(A,\beta) Q$, hence
\begin{equation} 
\det(B(A,\beta'))=\det(Q)^2 \det(B(A,\beta)) = \det(B(A,\beta)) 
\end{equation}
since $\det(Q) \in \{\pm 1\}$.  
\end{proof}

\begin{remark}
Strictly speaking, our definition of discriminant depends on the choice of representation $\lambda$.  One could also consider the right regular representation or indeed any faithful matrix representation of $A$.  Although these need not give the same answers, the proof below shows that they all satisfy a discriminant congruence.
\end{remark}

It will turn out to be crucial to our arguments in the next section to have $1$ as the first element of a basis.  

\begin{defin}
A \defi{unital basis} for $A$ is a basis $\beta=(e_1,\dots,e_n)$ with $e_1=1$, and a \defi{unitally framed ring of rank $n$} $(A,\beta)$ is a ring $A$ of rank $n$ equipped with a unital basis $\beta$.  
\end{defin}

\begin{prop} \label{prop:ranknunit}
Every ring of rank $n$ has a unital basis.
\end{prop}

\begin{proof}
Let $A$ be a ring of rank $n$ and let $\beta= (e_1,\dots,e_n)$ be a (not necessarily unital) basis for $A$. Then $1 = a_1e_1 + \dots + a_ne_n$ with $a_1,\dots,a_n\in\Z$.  

We first claim that $\gcd(a_1,\dots,a_n)=1$.  Indeed, using the multiplication table, we have
\begin{equation} 
e_1 = e_1 \cdot 1 = \sum_{k=1}^n a_k e_1 e_k = \sum_{k=1}^n a_k \left( \sum_{j=1}^n c_{1jk} e_j \right) = \sum_{j=1}^n \left(\sum_{k=1}^n a_k c_{1jk}\right) e_j. 
\end{equation}
Since $\beta$ is a basis, by the coefficient of $e_1$ we have $1=\sum_{k=1}^n a_k c_{11k}$.  We conclude that $\gcd(a_1,\dots,a_n)=1$.  
 
Consider the row vector $a \colonequals (a_1,\dots,a_n)$.  We claim that there exists (invertible) $Q \in \GL_n(\Z)$ such that $aQ=(1,0,\dots,0)$.  Although the proof of this claim can be found in many places, we give an argument here in order to be self-contained.  We proceed by induction.  The base case $n=1$ is immediate.  In general, by the extended Euclidean algorithm (B\'ezout relation), there exist $x_{n-1},x_n \in \Z$ such that $a_{n-1}x_{n-1}+a_n x_n = g \colonequals \gcd(a_{n-1},a_n)$.  Let 
\[ P \colonequals \begin{pmatrix} x_{n-1} & -a_n/g \\ x_n & a_{n-1}/g \end{pmatrix}; \]
then $\det(P)= 1$ so $P \in \SL_2(\Z)$, and the block matrix $\begin{pmatrix} I & 0 \\ 0 & P \end{pmatrix} \in \SL_n(\Z)$ has
\begin{equation} \label{eqn:a1ang}
(a_1,\dots,a_{n-1},a_n) \begin{pmatrix} I & 0 \\ 0 & P \end{pmatrix} = (a_1,\dots,a_{n-2},g,0) 
\end{equation}
still with $\gcd(a_1,\dots,g)=\gcd(a_1,\dots,a_{n-1},a_n)=1$.
Therefore by induction, there exists $Q \in \GL_{n-1}(\Z)$ such that
$(a_1,\dots,g)Q=(1,0,\dots,0)$, so multiplying \eqref{eqn:a1ang} by
$\begin{pmatrix} Q & 0 \\ 0 & 1 \end{pmatrix}$ gives the result.

From the claim, we have $aQ=(1,0,\dots,0)$ and so the first row of the
inverse $Q^{-1} \in \GL_n(\Z)$ is indeed $(a_1,\dots,a_n)$.  Now
consider the change of basis of $A$ provided by $Q^{-1}$: write
$Q^{-1}=(q_{ij})_{i,j=1}^{n}$ and let $f_i \colonequals \sum_{j=1}^n
q_{ij} e_j$ for $i=1,\dotsc,n$.  Then $f_1=1$, and so the elements
$f_i$ form a unital basis for $A$, as desired.
\end{proof} 

The next lemma, which follows an observation by Darij Grinberg, is
proved by direct computation.

\begin{lemma}\label{split}
The product $\Z\times A$ is a ring of rank $n+1$ with basis 
\[ \beta' = ((1,e_1), (0,e_1), \dots, (0, e_n)); \] 
$\beta'$ is unital if $\beta$ is unital; and $\disc(\Z \times A, \beta')=\disc(A,\beta)$.
\end{lemma}


\section{Stickelberger's discriminant theorem}\label{sec:over-Z}

In this section, we prove our main theorem, restated here for convenience.

\begin{theorem}\label{main-theorem}
 Let $A$ be a ring of rank $n$.  Then $\disc(A) \equiv 0,1 \pmod{4}$.  
\end{theorem}

The outline of the proof is as follows.  First, we study the properties of the Gram matrix of $A$, noting it has a certain property relating the first row and column to the diagonal; we call such Gram matrices \emph{tracelike}, and we prove the congruence more generally for tracelike matrices.  Second, we transform the symmetric matrix to one with even diagonal; from there, we \emph{expand by the adjugate} to establish the congruence.

\subsection*{Tracelike Gram matrices}

As a first step, consider the following well-known lemma.  We give a quick proof, to provide motivation and for completeness.  

\begin{lemma}\label{trace-square}
 We have $\Tr(M^2) \equiv \Tr(M)^2 \pmod{2}$ for all $M \in \M_n(\Z)$. 
\end{lemma}

\begin{proof}
 Let $M = (m_{ij})_{i,j=1}^n$. Then
 \[\Tr(M)^2 = \left(\sum_{i=1}^n m_{ii}\right)^2 = \sum_{i=1}^n m_{ii}^2 + 2\sum_{1\leq i < j \leq n} m_{ii}m_{jj},\]
 whereas
 \begin{align*}
 \Tr(M^2) &= \Tr\left(\sum_{j=1}^n m_{ij}m_{jk}\right)_{i,k=1}^n = \sum_{i=1}^n\sum_{j=1}^n m_{ij}m_{ji} \\
 &= \sum_{i=1}^n m_{ii}^2 + 2\sum_{1\leq i < j \leq n} m_{ij}m_{ji}.
 \end{align*}
So modulo $2$, both sums are congruent to $\sum_{i=1}^n m_{ii}^2$.
\end{proof}

In particular, Lemma \ref{trace-square} applies to the entries of the Gram matrices considered in the previous section (Definition \ref{discriminant}).

\begin{cor} \label{cor:gramabc}
Let $A$ be a ring of rank $n$, let $\beta=(e_1,\dots,e_n)$ be a unital basis for $A$, and let $B(A,\beta)=(b_{ij})_{i,j}$ be the Gram matrix of $(A,\beta)$.  Then $b_{11}=n$ and $b_{ii} \equiv b_{1i}^2 \pmod{2}$ for $i=2,\dots,n$.  
\end{cor}

\begin{proof}
Since $\beta$ is a unital basis we have $e_1=1$ so for all $i=1,\dots,n$ we have 
\begin{equation} \label{b1i}
b_{1i}=t(e_1,e_i)=\Tr(\lambda_\beta(e_i)). 
\end{equation}
Taking $i=1$ in \eqref{b1i} we get $b_{11}=\Tr(I)=n$, where $I \in \M_n(\Z)$ is the identity matrix.  For $i=2,\dots,n$, applying Lemma \ref{trace-square} and \eqref{b1i} gives
\[ b_{ii} = \Tr(\lambda_\beta(e_i)^2) \equiv \Tr(\lambda_\beta(e_i))^2 = b_{1i}^2 \pmod{2}. \qedhere \]
\end{proof}

Corollary \ref{cor:gramabc} isolates the key property that implies our desired congruence.  Accordingly, we make the following definition.  

\begin{defin}
A symmetric matrix $B=(b_{ij})_{i,j} \in \M_n(\Z)$ is \defi{tracelike} if $b_{11}=n$ and $b_{ii} \equiv b_{1i}^2 \pmod{2}$ for all $i=2,\dots,n$.
\end{defin}

The Gram matrix $B(A,\beta)$ of any framed ring $(A,\beta)$ of rank $n$
is a tracelike matrix by Corollary~\ref{cor:gramabc}.

\begin{ques}
Is every tracelike matrix the Gram matrix of an algebra in a unital basis?
\end{ques}

\subsection*{Symmetrizing}

We now proceed to study determinants of tracelike matrices.  Our proof consists first of a row reduction step to obtain a symmetric matrix with even diagonal; then we prove such matrices satisfy the desired congruence.  From now on, let $B=(b_{ij})_{i,j} \in \M_n(\Z)$ be a tracelike matrix.  

\begin{lemma}\label{reduce-to-symmetric-even-diagonal}
Let $B=(b_{ij})_{i,j} \in \M_n(\Z)$ be a tracelike matrix.  Suppose that $4 \mid n$, and for $i=2,\dots,n$ let $c_i \in \Z$ be such that $b_{ii}=b_{1i}^2+2c_i$.  
Let 
\[ C \colonequals 
\begin{pmatrix}
 n & b_{12} & b_{13} & \dots & b_{1n} \\
 b_{12} & 2c_2 & b_{23} - b_{12}b_{13} & \dots & b_{2n} - b_{12}b_{1n}\\
 b_{13} & b_{23} - b_{12}b_{13} & 2c_3 & \dots & b_{3n} - b_{13}b_{1n}\\
 \vdots & \vdots & \vdots & \ddots & \vdots \\
 b_{1n} & b_{2n} - b_{12}b_{1n} & b_{3n} - b_{13}b_{1n} & \dots & 2c_n
\end{pmatrix} \in \M_n(\Z). \]
Then $C$ is a symmetric matrix with diagonal entries in $2\Z$ and $\det(B) \equiv \det(C) \pmod{4}$.
\end{lemma}

\begin{proof}
We begin with 
\[ B = \begin{pmatrix}
 n & b_{12} & b_{13} & \dots & b_{1n} \\
 b_{12} & b_{12}^2 + 2c_2 & b_{23} & \dots & b_{2n}\\
 b_{13} & b_{23} & b_{13}^2 + 2c_3 & \dots & b_{3n}\\
 \vdots & \vdots & \vdots & \ddots & \vdots \\
 b_{1n} & b_{2n} & b_{3n} & \dots & b_{1n}^2 + 2c_n
\end{pmatrix}.
\]
Subtracting $b_{12}$ times the first row from the second, and $b_{13}$ times the first row from the third, and so on, we preserve the determinant:
\[ \det(B) = \det\begin{pmatrix}
 n & b_{12} & b_{13} & \dots & b_{1n} \\
 (1-n)b_{12} & 2c_2 & b_{23} - b_{12}b_{13} & \dots & b_{2n} - b_{12}b_{1n}\\
 (1-n)b_{13} & b_{23} - b_{12}b_{13} & 2c_3 & \dots & b_{3n} - b_{13}b_{1n}\\
 \vdots & \vdots & \vdots & \ddots & \vdots \\
 (1-n)b_{1n} & b_{2n} - b_{12}b_{1n} & b_{3n} - b_{13}b_{1n} & \dots & 2c_n
\end{pmatrix}.
\]
The result now follows since $4 \mid n$ so $1-n \equiv 1 \pmod{4}$.
\end{proof}

\subsection*{Expanding by adjugate}

Recall that the \defi{adjugate} of $A \in \M_n(\Z)$ is the transpose of the matrix of the cofactors of $A$, defined by
\begin{equation} 
\adj(A) \colonequals \left((-1)^{i+j} \det(A'_{ji})\right)_{i,j=1}^n,
\end{equation}
where $A'_{ij} \in \M_{n-1}(\Z)$ is the submatrix of $A$ obtained by removing the $i$th row and $j$th column.  We have
\begin{equation} \label{eqn:adjA}
A\adj(A)=\adj(A)A=\det(A)I
\end{equation}
as well as $\adj(A^{\matT})=\adj(A)^{\matT}$ and $\adj(cA)=c^{n-1}\adj(A)$ for $c \in \Z$.

\begin{proposition}\label{determinant-derivative}
Let $M,Q \in \M_n(\Z)$.  Then 
 \[ \det(M + 2Q) \equiv \det(M) + 2\Tr(\adj(M)Q)\pmod{4}. \]
\end{proposition}

\begin{proof}
Write $M = (m_{ij})_{i,j=1}^n$ and $Q=(q_{ij})_{i,j=1}^n$.  We begin with the expansion  
\begin{equation} \label{eqn:m2qexp1}
\det(M+2Q) = \sum_{\sigma \in S_n} (\sgn \sigma) \prod_{i=1}^n (m_{i\sigma(i)}+2q_{i\sigma(i)})
\end{equation}
where $S_n$ is the symmetric group of degree $n$.  Expanding out the right-hand side modulo $4$, for each $\sigma \in S_n$ we have
\begin{equation} \label{eqn:m2qexp2} 
\prod_{i=1}^n (m_{i\sigma(i)}+2q_{i\sigma(i)}) \equiv \prod_{i=1}^n m_{i\sigma(i)} + 2 \sum_{j=1}^n q_{j\sigma(j)} \prod_{\substack{i=1 \\ i \neq j}}^n m_{i\sigma(i)} \pmod{4}.
\end{equation}
Combining \eqref{eqn:m2qexp1}--\eqref{eqn:m2qexp2} and interchanging summations gives
\begin{equation} \label{eqn:m2qe3}
\det(M+2Q) \equiv \sum_{\sigma \in S_n} (\sgn \sigma) \prod_{i=1}^n m_{i\sigma(i)} + 2\sum_{j=1}^n \sum_{\sigma \in S_n} q_{j\sigma(j)} \prod_{\substack{i=1 \\ i \neq j}}^n m_{i\sigma(i)} \pmod{4},
\end{equation}
ignoring signs as we work with an even integer modulo $4$.  The first term is of course $\det(M)$.  For the second sum, for all $j,k$ we have 
\begin{equation} 
\det(M'_{jk}) = \pm \sum_{\substack{\sigma \in S_n \\ \sigma(j)=k}} (\sgn \sigma) \prod_{\substack{i=1 \\ i \neq j}}^n m_{i\sigma(i)}.
\end{equation}
Reorganizing the sum, working modulo $2$ so we may ignore signs, we obtain
\begin{equation} \label{eqn:m2qe4}
\begin{aligned}
\sum_{j=1}^n \sum_{\sigma \in S_n} q_{j\sigma(j)} \prod_{\substack{i=1 \\ i \neq j}}^n m_{i\sigma(i)} &\equiv
\sum_{j=1}^n \sum_{k=1}^n \sum_{\substack{\sigma \in S_n \\ \sigma(j)=k}} q_{jk} \prod_{\substack{i=1 \\ i \neq j}}^n m_{i\sigma(i)} 
\equiv \sum_{j=1}^n \sum_{k=1}^n q_{jk} \det(M'_{jk}) \\  & \equiv \sum_{j=1}^n \sum_{k=1}^n q_{jk} \adj(M)_{kj} \equiv \Tr(Q\adj(M)) \pmod{2}. 
\end{aligned}
\end{equation}
Plugging \eqref{eqn:m2qe4} into \eqref{eqn:m2qe3} then gives the result.  
\end{proof}

\begin{proof}[{Second proof of Proposition \textup{\ref{determinant-derivative}}}]
%
We extend our scope to real matrices and show that the identity holds when $M$ is invertible, and then for all matrices.  Let $M,Q \in \M_n(\R)$.  

First, using an indeterminate $x$ we have
\[ \det(M+xQ) = c_0(M,Q) + c_1(M,Q)x + \dots + c_n(M,Q)x^n \in \R[x]. \] 
For example, plugging in $x=0$ gives $c_0(M,Q)=\det(M)$ for all $M,Q$.  Moreover, $\det(I-xQ)$ is the reverse characteristic polynomial of $Q$, so $c_1(I,Q)=-\Tr(Q)$.  We define the map
\begin{equation}
\begin{aligned}
\M_n(\R) \times \M_n(\R) &\to \R \\
(M,Q) &\mapsto c_1(M,Q).
\end{aligned}
\end{equation}

Next, if $M \in \GL_n(\R)$ is invertible, we have
\begin{equation} \label{eqn:MTQ}
\begin{aligned}
\det(M+xQ) &= \det(M)\det(I+xM^{-1}Q) \\
&= \det(M)(1-\Tr(M^{-1}Q)x + h_{M,Q}(x))  \\
&= \det(M)-\Tr(\adj(M)Q)x + \det(M) x^2 h_{M,Q}(x)
\end{aligned}
\end{equation}
for some $h_{M,Q}(x) \in \R[x]$, using \eqref{eqn:adjA} which gives $\adj(M)=\det(M)M^{-1}$.  Thus $c_1(M,Q)=-\Tr(\adj(M)Q)$ for the set of matrices $M \in \GL_n(\R)$ which are dense with respect to the usual topology on $\M_n(\R) \simeq \R^{n^2}$.

Finally, the function $c_1(M,Q)$ is a polynomial in the entries of
$M,Q$ (as $M,Q$ range over $\M_n(\R)$) so continuous in these entries;
the same is true for $-\Tr(\adj(M)Q)$.  We just showed these functions
are equal whenever $\det(M) \neq 0$, so in fact they must equal for
all $M,Q$.  Restricting back to $M,Q \in \M_n(\Z)$, we have
\begin{equation}  \label{eqn:MTQoverZZ}
\begin{aligned}
\det(M+xQ) &= \det(M) - \Tr(\adj(M)Q) x  \\
&\qquad\qquad + c_2(M,Q)x^2 + \dots + c_n(M,Q)x^n \in \Z[x], 
\end{aligned}
\end{equation}
the resulting polynomial visibly having integer coefficients.  Plugging in $x=2$ into \eqref{eqn:MTQoverZZ} then gives the result.
\end{proof}

\begin{remark}
Many linear algebra statements can be proven in the same manner as the second proof, using the method of \emph{universal polynomials}, where the entries of the matrices are left as indeterminates.  If instead of a congruence, once wishes to prove an equality, then it is enough to do so over the field $\Q(x_{ij})_{i,j}$, where now the determinant is a nonzero polynomial, so invertible.  For example, the Cayley--Hamilton theorem may be proven this way.
\end{remark}

\subsection*{Determinants of even symmetric matrices}

We are now ready for the second key step in the proof.  We will use the fact that the determinant of every skew-symmetric matrix $A$ has a canonical square root called its \defi{pfaffian} $\pf(A)$; see, for example, Stembridge \cite[Proposition 2.2]{stembridge:pfaffians}.

\begin{prop} \label{prop:detC4}
Let $C \in \M_n(\Z)$ be a symmetric matrix with diagonal entries in $2\Z$.  Suppose $4 \mid n$, and let $U$ be the matrix obtained from the upper-triangular part of $C$ and half its diagonal.  Then $C=U+U^{\matT}$ and
\[ \det(C) \equiv \det(U-U^{\matT})=\pf(U-U^{\matT})^2 \equiv 0,1 \pmod{4}. \]
\end{prop}

For this proposition, we need a lemma.

\begin{lem} \label{lem:uut}
Let $M \in \M_n(\Z)$ with $2 \mid n$.  Then 
\[ 2\Tr(\adj(M-M^{\matT})M^{\matT}) = -n\det(M-M^{\matT}). \]
\end{lem}

\begin{proof}
Let $r \colonequals \Tr(\adj(M-M^{\matT})M^{\matT})$.  
Taking the transpose and recalling the properties of the adjugate,
\begin{equation}
 \begin{aligned}
 r &= \Tr(M(\adj(M-M^{\matT}))^{\matT}) = \Tr(M\adj(M^{\matT} - M))\\
  &= \Tr(M(-1)^{n-1}\adj(M-M^{\matT}))= -\Tr(\adj(M-M^{\matT})M).
 \end{aligned}
 \end{equation}
Adding back $r$, by linearity of trace we have
\begin{align*}
 2r &= \Tr(\adj(M-M^{\matT})M^{\matT}) - \Tr(\adj(M-M^{\matT})M)\\
 &= \Tr(\adj(M-M^{\matT})(M^{\matT}-M)) = \Tr(-\det(M-M^{\matT})I) \\
 &= -n\det(M-M^{\matT})
\end{align*}
proving the claim.
\end{proof}

\begin{proof}[Proof of Proposition~\textup{\ref{prop:detC4}}]
We have $C=U+U^{\matT}=(U-U^{\matT}) + 2U^{\matT}$.  By Proposition~\ref{determinant-derivative}, we have
\begin{equation} 
\det(C) \equiv \det(U-U^{\matT}) + 2\Tr(\adj(U-U^{\matT})U^{\matT}) \pmod{4}. 
\end{equation}
Since $U-U^{\matT}$ is a skew-symmetric matrix, we have 
\[
\det(U-U^{\matT})=\pf(U-U^{\matT})^2 \equiv 0,1\pmod{4}.
\]  
By Lemma \ref{lem:uut} we have
\begin{equation} 
2\Tr(\adj(U-U^{\matT})U^{\matT}) \equiv 0 \pmod{4}
\end{equation}
and the result follows.
\end{proof}

\begin{remark}
The hypothesis $4 \mid n$ in Proposition \ref{prop:detC4} is
necessary. Indeed, for $n=1$ we could take $\det\begin{pmatrix}
2 \end{pmatrix}=2$, for $n=2$ we could take $\det\begin{pmatrix} 0 & 1
\\ 1 & 0 \end{pmatrix}=-1$, and for $n=3$ we could take the block
matrix obtained from these two.
\end{remark}

\subsection*{Proof conclusion}

With these ingredients in hand, we now prove our main theorem.

\begin{proof}[Proof of Theorem~\textup{\ref{main-theorem}}]
Let $A$ be a ring of rank $n$.  Replacing $A$ by $A \times \Z^r$ if necessary, by Lemma~\ref{split} we may suppose without loss of generality that $4 \mid n$.  By Proposition \ref{prop:ranknunit}, $A$ has a unital basis $\beta$, so by Corollary \ref{cor:gramabc}, the Gram matrix $B=B(A,\beta)$ is tracelike.  Then by Lemma~\ref{reduce-to-symmetric-even-diagonal}, there exists a symmetric matrix $C$ with even diagonal such that $\det(B) \equiv \det(C)\pmod 4$.  Putting these together and applying Proposition \ref{prop:detC4}:
\[ \disc(A)=\det(B) \equiv \det(C) \equiv 0,1 \pmod 4 \]
as desired.
\end{proof}

\subsection*{Generalizations}

The proof of our main result used just techniques from linear algebra.  Accordingly, it immediately generalizes to a wider context, allowing an arbitrary commutative base ring.  

Let $R$ be a commutative ring (with $1$).  An \defi{$R$-algebra} is a
ring $A$ (with $1$), not necessarily commutative, equipped with a ring
homomorphism $R \to A$ whose image lies in the center of $A$.  For the
$R$-algebras considered in this section, we will suppose that the map
$R \hookrightarrow A$ is injective, so that we may identify $R$ with
its image $R1 \subseteq A$.  An $R$-algebra $A$ is \defi{free of rank
$n$} if $A \simeq R^n$ as $R$-modules, i.e., $A$ has an $R$-basis
$\beta=(e_1,\dots,e_n)$.

The rest of the definitions and results in Section \ref{sec:notation}
generalize, with only two adjustments.  First, in contrast to
Lemma~\ref{lem:welldefdisc}, we only obtain a well-defined
discriminant $\disc A \in R/R^{\times 2}$, the set of elements of $R$
up to squares of units in $R$, as $\det(\GL_n(R))=R^\times$.  Second,
in contrast to Proposition~\ref{prop:ranknunit}, we do not know
whether unital bases exist for an arbitrary free $R$-algebra.
However, we may always reduce to working with a unitally framed
algebra by invoking the following, which is a direct generalization of
Lemma~\ref{split}.

\begin{lem}
If $(A,\beta)$ is a framed $R$-algebra of rank $n$ with $\beta=(e_1,\dots,e_n)$, then $A' \colonequals R \times A$ has a unital framing
\begin{equation} \label{eqn:yupyup}
\beta'=((1,1),(0,e_1),\dots,(0,e_n)).  
\end{equation}
Furthermore, we have $\disc(R \times A,\beta')=\disc(A,\beta)$ in
$R/R^{\times 2}$.
\end{lem}



%
%
%

With these in mind, the same proof gives the following theorem.

\begin{theorem}\label{PID-theorem}
Let $R$ be a commutative ring and let $A$ be a free $R$-algebra of rank $n$.  Then 
\[ \disc(A,\beta) \equiv \discpf(A',\beta')^2 \pmod{4} \]
where $A'=R \times A$ and $\beta'$ is as in \textup{\eqref{eqn:yupyup}}.
\end{theorem}

\begin{remark}
We can also go a bit farther, arguing \emph{locally}.  An $R$-module $M$ is said to have some property \defi{(Zariski-)locally} if there exist $r_1,\dots,r_m\in R$ generating the unit ideal $R$ such that the localization $M[r_i^{-1}]$ has that property as an $R[r_i^{-1}]$-module for all $i=1,\dots,m$.  
In particular, we can speak of an $R$-module $M$ being \defi{locally free of rank $n$}.
The same arguments then show that if $R$ is a commutative ring and $A$ is an $R$-algebra that is locally free of rank $n$ as an $R$-module, then $\disc(A)$ is locally a square modulo $4$.  
\end{remark}


%
%
%
%
%
%

\section{Discriminant pfaffian} \label{sec:discpf}

In this section, we refine the result of the previous section by giving an explicit, combinatorial expression for our ``square root modulo 4'' obtained from pfaffians.

To begin, recall that a \defi{perfect matching} $P$ on a set $J$ is a partition of $J$ into subsets of cardinality $2$.

\begin{defin}\label{def-dpf}
Let $B=(b_{ij})_{i,j} \in \M_n(\Z)$ be tracelike.  Define the \defi{discriminant pfaffian} of $B$ by
\[\dpf(B) \colonequals \sum_{\substack{J\subseteq \{2,\dots,n\}\\ \#J\text{ even}}} \sum_{\substack{\text{perfect} \\ \text{matchings}\\P\text{ on }J}}\Biggl(\prod_{\{i,j\}\in P} b_{ij}\Biggr)\Biggl(\prod_{k\in\{2,\dots,n\}\smallsetminus J} b_{1k}\Biggr).\]
(If $P=\emptyset$, by convention the empty product is defined to be $1$.)

If $(A,\beta)$ is a unitally framed ring of rank $n$, define its \defi{discriminant pfaffian} by
\[ \dpf(A,\beta) = \dpf(B(A,\beta)). \]
\end{defin}



The value of $\dpf(B)$ for the first few values of $n$ are as follows:
\[\begin{array}{c|c}
n & \dpf(B)  \\ \hline
1 & 1\\
2 & b_{12}\\
3 & b_{12}b_{13} + b_{23}\\
4 & b_{12}b_{13}b_{14} + b_{23}b_{14} + b_{24}b_{13} + b_{34}b_{12}\\
\end{array}\]
In general the number of terms in $\discpf(A, \beta)$ is given by the number of involutions on a set of $n-1$ letters \cite[Sequence A000085]{OEIS}.

\begin{theorem}\label{thm-dpf}
 Let $B \in \M_n(\Z)$ be tracelike.  Then $\det(B) \equiv \dpf(B)^2 \pmod{4}$.  
\end{theorem}

\begin{exm}
For example, if $n=2$ we have
\[ B = \begin{pmatrix} 2 & b_{12} \\ b_{12} & b_{12}^2 + 2c_2 \end{pmatrix} \]
for some $c_2 \in \Z$, so 
\[ \det B = 2b_{12}^2+4c_2 - b_{12}^2 \equiv b_{12}^2 = \dpf(B)^2 \pmod{4}. \]
\end{exm}

Before proceeding with the proof, we motivate the discriminant pfaffian using the modern proof of Stickelberger's theorem.

\begin{rmk} \label{rmk:PNPN}
  Let $K$ be a number field with ring of integers $\Z_K$ and integral
  basis $\alpha_1,\dots,\alpha_n$.  Letting
  $\sigma_1,\dots,\sigma_n \colon K \hookrightarrow \C$ be the
  distinct embeddings of $K$ into $\C$, we consider the $n \times n$
  matrix of complex numbers $E \colonequals (\sigma_i(\alpha_j))_{i,j}$.  Then
  $B = E^\matT E$ (see Marcus \cite[Theorem 6]{Marcus:nf}) and so $\disc \Z_K = \det(B) = \det(E)^2$.  
  (The definition in \eqref{eqn:BTralp} has the virtue that it expresses the discriminant as the determinant of a matrix of integers.)
   
  Letting $P$ and $N$ be the sum
  of terms in the expansion of $\det(E)$ involving even and odd
  permutations, respectively, the standard proof of Stickelberger's
  discriminant theorem is to write
  \begin{equation}
   \disc \Z_K = \det(\sigma_i(\alpha_j))_{i,j}^2 = (P-N)^2 =
    (P+N)^2-4PN; 
    \end{equation}
    by construction, the elements $P+N,PN$ are algebraic integers, and
    by Galois theory they belong to $\Q$, hence $P+N,PN \in \Z$
    and the result follows.  With this in mind, a natural square
  root of the discriminant modulo 4 is $P+N$, which is equal to the
  \emph{permanent} of the matrix $E$.    
  This permanent agrees with the 
  discriminant pfaffian modulo $2$ by Theorem \ref{thm-dpf}.
\end{rmk}

\begin{proof}[{Proof of Theorem~\textup{\ref{thm-dpf}}}]
 We first consider the case that $4 \mid n$.  Combining Lemma \ref{reduce-to-symmetric-even-diagonal} and Proposition \ref{prop:detC4}, we have
 \begin{equation} 
 \det(B) \equiv \pf(U-U^{\matT})^2 \pmod{4} 
 \end{equation}
 where
\begin{equation} 
U - U^\matT = \begin{pmatrix}
  0 & b_{12} & b_{13} & \dots & b_{1n}\\
  -b_{12} & 0 & b_{23} - b_{12}b_{13} & \dots & b_{2n} - b_{12} b_{1n}\\
  -b_{13} & -(b_{23} - b_{12}b_{13}) & 0 & \dots & b_{3n} - b_{13} b_{1n}\\
  \vdots & \vdots & \vdots & \ddots & \vdots\\
  -b_{1n} & -(b_{2n} - b_{12}b_{2n}) & -(b_{3n} - b_{13}b_{1n}) & \dots & 0
 \end{pmatrix}.
\end{equation}
Since $x \equiv y \pmod{2}$ implies $x^2 \equiv y^2 \pmod{4}$ for all $x,y \in \Z$, it suffices to show that
\begin{equation} \label{eqn:BUUT}
\dpf(B) \equiv \pf(U-U^{\matT}) \equiv \pf(U+U^{\matT}) \pmod{2}. 
\end{equation}
In particular, we can ignore signs throughout.

To prove \eqref{eqn:BUUT}, we write $U-U^\matT \equiv B'+B'' \pmod{2}$ where
\begin{align*} 
B' &\colonequals \begin{pmatrix}
  0 & 0 & 0 & \dots & 0\\
  0 & 0 & b_{23} & \dots & b_{2n}\\
  0 & b_{23} & 0 & \dots & b_{3n} \\
  \vdots & \vdots & \vdots & \ddots & \vdots\\
  0 & b_{2n}  & b_{3n}  & \dots & 0
 \end{pmatrix} \\
 B'' &\colonequals 
 \begin{pmatrix}
  0 & b_{12} & b_{13} & \dots & b_{1n}\\
  b_{12} & 0 &  b_{12}b_{13} & \dots & b_{12} b_{1n}\\
  b_{13} & b_{12}b_{13} & 0 & \dots & b_{13} b_{1n}\\
  \vdots & \vdots & \vdots & \ddots & \vdots\\
  b_{1n} &  b_{12}b_{2n} & b_{13}b_{1n} & \dots & 0
 \end{pmatrix}. 
\end{align*}

By Stembridge \cite[Lemma 4.2(a)]{stembridge:pfaffians},
\[
\pf(B' + B'') \equiv \sum_{\substack{J\subseteq\{1,\dots,n\}\\\#J\text{ even}}} \pf(B'_J) \pf(B''_{J^c}) \pmod{2}
\]
where $B'_J$ is the submatrix of $B'$ obtained by keeping only the entries in rows and columns indexed by elements of $J$, and $B''_{J^c}$ is similarly the matrix of entries in $B''$ whose row and column indices are \emph{not} in $J$.

To evaluate $\pf(B'_J)$ modulo $2$, note first that if $1\in J$, then $B'_J$ contains a row of zeros, so its determinant (and therefore its pfaffian) vanishes. Otherwise, the $ij$th entry of $B'$ is just $b_{ij}$ for $i > j$, so by the usual pfaffian formula in terms of perfect matchings (see the definition in \cite[p. 102]{stembridge:pfaffians}), we have
\begin{equation}
\pf(B'_J) \equiv \sum_{P\text{ on }J}\prod_{\{i,j\}\in P} b_{ij} \pmod{2}, 
\end{equation}
the sum over perfect matchings $P$ on $J$.  Meanwhile, given a subset $J\subseteq\{2,\dots,n\}$, each perfect matching on $J^c$ contributes the same product to $\pf(B''_{J^c})$, namely $\prod_{k\in \{2, \dots, n\}\setminus J} b_{1k}$. Since there is an odd number of perfect matchings on any even-cardinality set, modulo $2$ we have $\pf(B'_{J^c}) \equiv \prod_{k\in \{2, \dots, n\}\setminus J} b_{1k} \pmod{2}$. Now we just put this altogether:
\begin{equation}
\begin{aligned}
 \pf(U-U^\matT) &\equiv \pf(B' + B'') \equiv \sum_{\substack{J\subseteq\{1,\dots,n\}\\\#J\text{ even}}} \pf(B'_J) \pf(B''_{J^c})\\
 &\equiv \sum_{\substack{J\subseteq\{2,\dots,n\}\\\#J\text{ even}}} \sum_{\substack{P\text{ on }J}}\Biggl(\prod_{\{i,j\}\in P} b_{ij}\Biggr) \Biggl(\prod_{k\in \{2, \dots, n\}\smallsetminus J} b_{1k}\Biggr)
 \\
 &\equiv \dpf(A,\beta) \pmod{2}.
\end{aligned}
\end{equation}

Having proven it for all $n$ such that $4 \mid n$, we finish by a reverse induction, showing that if statement holds for $n \in \Z_{\geq 2}$ then it holds for $n-1$.  Since every positive integer $n$ is less than or equal to a multiple of $4$, this will prove the theorem.  
Let $B_{n-1} \in \M_{n-1}(\Z)$ be tracelike.  Let $B_n=\begin{pmatrix} B_{n-1} & 0 \\ 0 & 1 \end{pmatrix} \in \M_{n}(\Z)$ be the block matrix formed from $B_{n-1}$ and $\begin{pmatrix} 1 \end{pmatrix}$.  Then $\det(B_n)=\det(B_{n-1})$.  If we add the last row to the first row, then add the last column to the first column, we obtain 
\begin{equation} 
B_n' \colonequals \begin{pmatrix} B_{n-1}' & v^{\matT} \\ v & 1 \end{pmatrix}
\end{equation}
where $v=(1,0,\dots,0)$ and $B_{n-1}'$ is the matrix obtained from $B_{n-1}$ by adding $1$ to $b_{11}$.  Now $B_n'$ is tracelike!  By the inductive hypothesis, we have
\begin{equation} 
\det(B_{n-1})=\det(B_n')\equiv \dpf(B_n')^2 \pmod{4}. 
\end{equation}
Now the discriminant pfaffian $\dpf(B_n')$ is obtained from substituting $b_{1n}=1$ and $b_{in}=0$ for $i=2,\dots,n-1$.  To evaluate, we look at Definition \ref{def-dpf}.  For every term with $n \in J$, any perfect matching $P$ on $J$ has $\{i,n\}$ in $P$ with $i \in \{2,\dots,n-1\}$ and therefore such a term vanishes.  On the other hand, every term with $n \not \in J$ corresponds to $J \subseteq \{2,\dots,n-1\}$ with final term $\prod_{k \in \{2,\dots,n-1\} \smallsetminus J} b_{1k}$ since $b_{1n}=1$.  
\end{proof}

\nocite{}
\bibliographystyle{amsalpha}
\bibliography{stickelberger}

\begin{thebibliography}{Hur1896}

\bibitem[Bae81]{baeza:discriminant}
Ricardo Baeza, \emph{Discriminants of polynomials and of quadratic forms},
  J.\ Algebra \textbf{72} (1981), no.~1, 17--28.

\bibitem[Bha06]{bhargava:ICM}
Manjul Bhargava,
\emph{Higher composition laws and applications},
International {C}ongress of {M}athematicians.\ {V}ol. {II},
Eur. Math. Soc., Z\"urich, 2006, 271--294.

\bibitem[Ber76]{berlekamp:discriminant}
E.R.~Berlekamp, \emph{An analog to the discriminant over fields of
  characteristic two}, J.\ Algebra \textbf{38} (1976), no.~2, 315--317.

\bibitem[Bou98]{bourbaki:locally-free}
Nicolas Bourbaki, \emph{El\'ements de math\'ematique. Alg\`ebre commutative. Chapitres 1 \`a 4}, vol. 1,
Springer, 1998.

\bibitem[BG16]{biesel_gioia:discriminant}
Owen Biesel and Alberto Gioia, \emph{A new discriminant algebra construction},
  Documenta Math. \textbf{21} (2016), 1051--1088.

\bibitem[Cox]{cox:stickelberger}
David~A. Cox, \emph{Stickelberger and the eigenvalue theorem}, preprint, 2020, \texttt{arXiv:2007.12573}.

\bibitem[Gri17]{Darij}
Darij Grinberg (\url{https://mathoverflow.net/users/2530/darij-grinberg}), Is the discriminant of a free (as a module) $R$-algebra always congruent to a square modulo $4$?, URL (version: 2017-04-13): \url{https://mathoverflow.net/q/257889}.

\bibitem[Har12]{harrer}
Daniel Harrer, \emph{Parametrization of cubic rings}, Diplomarbeit, Universit\"at M\"unchen, 2012.

\bibitem[Hur1896]{Hurwitz:uberdie}
Adolf Hurwitz, \emph{\"Uber die Zahlentheorie der Quaternionen}, Nachrichten der Gesellschaft der Wissenschaften zu G\"ottingen, 1896, 314--340.

\bibitem[Lam06]{Lam06}
T.~Y.~Lam, \emph{Serre's problem on projective modules}, Springer Mono.~Math., Springer--Verlag, Berlin, 2006.

\bibitem[Mar18]{Marcus:nf}
Daniel A.\ Marcus, \emph{Number fields}, 2nd ed., Universitext, Springer Nature, Cham, 2018.

\bibitem[Mar90]{marcus}
Marvin Marcus, \emph{Determinants of sums},
Coll. Math. J. \textbf{21} (1990), no.\ 2, 130--135.

\bibitem[Mar89]{martinet:discriminant}
Jacques Martinet, \emph{Les discriminants quadratiques et la congruence de
  Stickelberger}, J.~Th\'eorie Nombres Bordeaux \textbf{1}
  (1989), no.~1, 197--204.

\bibitem[Neu99]{neukirch}
J\"urgen Neukirch, \emph{Algebraic number theory}, Springer, Berlin,
  1999.

\bibitem[Sch29]{schur:stickelberger}
I.~Schur, \emph{Elementarer Beweis eines Satzes von L.~Stickelberger},
  Math.~Z.\ \textbf{29} (1929), no.~1, 464--465.
  
\bibitem[OEIS]{OEIS}
N. J. A. Sloane, editor, \emph{The On-Line Encyclopedia of Integer Sequences}, published electronically at \url{https://oeis.org}, 2021, Sequence A000085
  
  \bibitem[Ste90]{stembridge:pfaffians}
John~R. Stembridge, \emph{Nonintersecting paths, {P}faffians, and plane
  partitions}, Adv. Math. \textbf{83} (1990), no.~1, 96--131. \MR{1069389}

\bibitem[Sti98]{stickelberger:discriminant}
L.~Stickelberger, \emph{{\"U}ber eine neue {E}igenschaft der {D}iskriminanten
  algebraischer {Z}ahlk {\"o}rper}, {V}erhandlungen des ersten internationalen
  {M}athematiker-{K}ongresses, {Z\"u}rich 1897 (Teubner, Leipzig) (F.~Rudio,
  ed.), 1898, p.~182--193.
  
\bibitem[Voi21]{Voight:quat}
John Voight, \emph{Quaternion algebras}, Grad.\ Texts in Math., vol.~288, Springer, Cham, 2021.

\bibitem[Was82]{Washington}
Lawrence C.\ Washington, \emph{Introduction to cyclotomic fields}, Grad.\ Texts.\ in Math., vol.~83, Springer, New York, 1982.

\end{thebibliography}

\providecommand{\bysame}{\leavevmode\hbox to3em{\hrulefill}\thinspace}
\providecommand{\MR}{\relax\ifhmode\unskip\space\fi MR }
\providecommand{\MRhref}[2]{%
  \href{http://www.ams.org/mathscinet-getitem?mr=#1}{#2}
}
\providecommand{\href}[2]{#2}

\end{document}